\documentclass[11pt,oneside,a4paper]{amsart}
\usepackage[british]{babel}

\usepackage[dvips]{graphicx}
\usepackage{amscd}
\usepackage{amsmath}
\usepackage{amsfonts}
\usepackage{amssymb}
\usepackage{mathrsfs}
\usepackage{comment}
\excludecomment{mycomment} 
\usepackage{color}
\usepackage[usenames,dvipsnames,table,xcdraw]{xcolor}
\usepackage{pdfsync}
\usepackage{bbm}
\usepackage{dsfont}
\usepackage[colorlinks=true]{hyperref}
\usepackage{booktabs}
\usepackage{float}
\usepackage{wrapfig}
\usepackage{caption}
\usepackage{subcaption}
\usepackage{longtable}
\usepackage{listings}
\usepackage[T1]{fontenc}
\usepackage[scaled]{beramono}
\usepackage{tikz}
\usepackage{pgffor}
\usepackage[all]{xy}
\usepackage{bbold}
\usepackage{enumerate}
\usepackage{lmodern}
\usepackage{epstopdf}

\definecolor{trp}{rgb}{1,1,1}

\definecolor{red}{rgb}{1,0,.2}

\usepackage{amsthm}
\theoremstyle{plain}
\newtheorem{theorem}{Theorem}[section]

\newtheorem{corollary}[theorem]{Corollary}

\newtheorem{definition}[theorem]{Definition}
\newtheorem{example}[theorem]{Example}

\newtheorem{lemma}[theorem]{Lemma}

\newtheorem{proposition}[theorem]{Proposition}

\newtheorem{remark}[theorem]{Remark}

\newtheorem*{terminology}{Terminology}

\numberwithin{equation}{section}

\newcommand{\iiv}{\overline{\imath}}

\newcommand*{\arabicdec}[1]{\the\numexpr\value{#1}\relax}

\newcommand{\nice}{limit-irreducible}

\usepackage{anysize}
\papersize{24.5cm}{16.2006cm}

\marginsize{1cm}{1cm}{0cm}{0.5cm}

\usepackage{caption}

\definecolor{blue}{rgb}{0,0,1}

\definecolor{red}{rgb}{1,0,.7}

\definecolor{myred1}{RGB}{255, 0, 0}

\usepackage{tikz}
\usetikzlibrary{cd,decorations.pathreplacing,positioning,arrows}

\usepackage{xr}

\begin{document}
\title[ ]{Continuity of the natural dimension of piecewise linear iterated function systems}

\author{R. D\'aniel Prokaj}
\address{ R. D\'aniel Prokaj,  Alfréd Rényi Institute of Mathematics, 
Reáltanoda u. 13-15., 1053 Budapest, Hungary} 
\email{prokajrd@math.bme.hu}

\author{Peter Raith}
\address{Peter Raith, Fakultät für Mathematik, Universität Wien, 
Oskar-Morgenstern-Platz 1, 1090 Wien, Austria}
\email{peter.raith@univie.ac.at}


\thanks{
\\ \indent 
The research of the first author was partially supported by National Research, Development and Innovation Office - NKFIH, Project K142169. This work was partially supported by the grant Stiftung Aktion Österich Ungarn 103öu6. 
}  

\begin{abstract}
We consider iterated function systems on the real line that consist of continuous, piecewise linear functions. We show that typically the natural dimension of these systems changes continuously with respect to the parameters that define the system. As an application of this property, we prove a result on the positivity of the Lebesgue measure of the attractor.
\end{abstract}
\date{\today}

\maketitle


\thispagestyle{empty}
\section{Introduction}\label{md20}

Iterated Function Systems (IFS) on the line consist of finitely many strictly contracting self-mappings of $\mathbb{R}$. In this paper, we consider IFSs consisting of piecewise linear functions. We always assume that the functions are continuous, piecewise linear, strongly contracting with non-zero slopes, and that the slopes can only change at finitely many points.

It was proved by Hutchinson \cite{hutchinson1981fractals}
that for every IFS $\mathcal{F}=\left\{f_k\right\}_{k=1}^{m}$ there is a unique non-empty compact set $\Lambda$ which is
called the attractor of the IFS $\mathcal{F}$ and defined by
 \begin{equation}\label{cr65}
   \Lambda=\bigcup\limits_{k=1}^{m}f_k(\Lambda).
 \end{equation}

For every IFS $\mathcal{F}$ there exists a unique ``smallest'' non-empty compact interval $I$  which is sent into itself by all the mappings of $\mathcal{F}$:
\begin{equation}\label{cr22}
  I:=\bigcap\left\{ J \big\vert\:
  J\subset \mathbb{R}\mbox{ compact interval} :
  f_k(J)\subset J, \forall k\in[m]
  \right\},
\end{equation}
where $[m]:=\left\{ 1,\dots  ,m \right\}$. To guarantee that $I$ is a non-degenerate interval, when the attractor of $\mathcal{F}$ is a single point we set 
\begin{equation*}
  I:=\left[\phi-\frac{1}{2},\phi+\frac{1}{2}\right],
\end{equation*}
where $\phi$ is the common fixed point of the functions $f_1,\dots, f_m$.
It is easy to see that 
\begin{equation}\label{cr15}
  \Lambda=\bigcap\limits_{n=1}^{\infty}
  \bigcup\limits_{(i_1,\dots  ,i_n)\in[m]^n}
  I_{i_1\dots  i_n},
\end{equation}
where $I_{i_1\dots  i_n}:=
f_{i_1\dots  i_n}(I)$ are the \textbf{cylinder intervals}, and we use the common shorthand notation $f_{i_1\dots  i_n}:=f_{i_1}\circ\cdots \circ f_{i_n}$ for an
$(i_1,\dots  ,i_n)\in [m]^n$. 

We define the \textbf{natural pressure function} as  
\begin{equation}\label{cr64}
    \Phi(s):=\limsup_{n\rightarrow\infty}\frac{1}{n}\log \sum_{i_1\dots i_n} |I_{i_1\dots i_n}|^s.
\end{equation}
In \cite{barreira1996non}, Barreira showed that $\Phi(s): \mathbb{R}_+\to \mathbb{R}$ is a strictly decreasing function with $\Phi(0)>0$ and $\lim_{s\to\infty} \Phi(s)=-\infty$. Hence we can define the \textbf{natural dimension} of $\mathcal{F}$ as  
\begin{equation}\label{cr61}
    s_{\mathcal{F}}:=(\Phi)^{-1}(0).
\end{equation}
He also proved that the upper box dimension, and hence the Hausdorff dimension, is always smaller or equal to the natural dimension.

Let $\mathcal{F}=\left\{f_k\right\}_{k=1}^{m}$ be a CPLIFS and 
$I \subset \mathbb{R}$ be the compact interval defined in \eqref{cr22}. For any $k\in [m]$
let $l(k)$ be the number of breaking points $\left\{b_{k,i}\right\}_{i=1}^{l(k)}$ of $f_k$.
They determine the $l(k)+1$ open intervals of linearity
$\left\{J_{k,i}\right\}_{i=1}^{l(k)+1}$.
A more detailed description of the parameter space is given in Section \ref{pm75}.

A continuous piecewise linear IFS $\mathcal{F}=\{f_k\}_{k=1}^m$ is uniquely determined by the slopes $\{\rho_{k,1},\dots,\rho_{k,l(k)+1}\}_{k=1}^m$, the breaking points $\{b_{k,1},\dots,$ $b_{k,l(k)}\}_{k=1}^m$ and the vertical translations $\{f_k(0)\}_{k=1}^m$ of its functions. We often refer to the latter two as the translation parameters of $\mathcal{F}$. 
We rely on the following notion of typicality.

\begin{terminology}\label{z99} 
  Given a property which is meaningful for all CPLIFSs.
  We say that this property is $\pmb{\dim_{\rm P}}$\textbf{-typical} if the set of translation parameters for which it does not hold has less than full packing dimension, for any fixed vector of slopes.
\end{terminology}

We write $S_{k,i}$ for the contracting similarity on $\mathbb{R}$ that satisfies
$S_{k,i}|_{J_{k,i}}\equiv f_k|_{J_{k,i}}$, 
and define $\left\{\rho_{k,i}\right\}_{k\in[m],i\in [l(k)+1]}$ and $\left\{t_{k,i}\right\}_{k\in[m],i\in [l(k)+1]}$ such that
\begin{equation}\label{cr56}
  S_{k,i}(x)=\rho_{k,i}x+t_{k,i}.
\end{equation}
We say that $\mathcal{S}_{\mathcal{F}}:=\left\{S_{k,i}\right\}_{k\in[m],i\in [l(k)+1]}$ is the \textbf{self-similar IFS generated by the CPLIFS $\mathcal{F}$}.

Now we define a separation condition for self-similar iterated function systems on the line.
Let $g_1(x)=\rho_1x+\tau_1$ and $g_2(x)=\rho_2x+\tau_2$ be two similarities on $\mathbb{R}$ with $\rho_1,\rho_2\in\mathbb{R}\setminus \{0\}$ and $\tau_1,\tau_2\in\mathbb{R}$.
We define the distance of these two functions as 
\begin{equation}\label{md17}
\mathrm{dist}(g_1,g_2):=\begin{cases}
  \vert \tau_1-\tau_2\vert \mbox{, if } \rho_1=\rho_2; \\
  \infty \mbox{, otherwise.}
\end{cases}
\end{equation}
\begin{definition}\label{md16}
Let $\mathcal{F}=\{f_k(x)\}_{k=1}^m$ be a self-similar IFS on $\mathbb{R}$. We say that $\mathcal{F}$ satisfies the \textbf{Exponential Separation Condition (ESC)} if there exists a $c>0$ and a strictly increasing sequence of natural numbers $\{n_l\}_{l=1}^{\infty}$ such that 
\begin{equation}\label{md15}
  \mathrm{dist}(f_{\mathbf{i}},f_{\mathbf{j}})\geq c^{n_l}\mbox{, for all } l>0 \mbox{ and for all } \mathbf{i},\mathbf{j}\in [m]^{n_l}, \mathbf{i}\neq\mathbf{j}. 
\end{equation}
\end{definition}

\begin{theorem}[{\cite[Theorem~1.4]{prokaj2022fractal}}]\label{md46}
Let $\mathcal{F}$ be a CPLIFS with generated self-similar system $\mathcal{S}$ and attractor $\Lambda$. If $\mathcal{S}$ satisfies the ESC, then
\begin{equation}\label{md45}
  \dim_{\rm H}\Lambda = \dim_{\rm B}\Lambda = \min \{1, s_{\mathcal{F}}\}.
\end{equation} 
\end{theorem}
Hochman proved that the ESC is a $\dim_{\rm P}$-typical property of self-similar IFS \cite[Theorem~1.10]{Hochman_2015}. Prokaj and Simon extended this result by showing that it is a $\dim_{\rm P}$-typical property of a CPLIFS that the generated self-similar system satisfies the ESC \cite[Fact~4.1]{prokaj2021piecewise}.

It follows that typically the Hausdorff dimension of the attractor of a CPLIFS is equal to the minimum of the natural dimension and $1$.
\begin{theorem}[{\cite[Theorem~1.2]{prokaj2022fractal}}]\label{md10}
  We write $\Lambda _{\mathcal{F}}$ for the attractor of a CPLIFS  $\mathcal{F}$.
  Then the following property is $\dim_{\rm P} $-typical:
  \begin{equation}\label{md09}
    \dim_{\rm H}\Lambda_{\mathcal{F}} = \min \{1, s_{\mathcal{F}}\}.
  \end{equation}
\end{theorem}

In this paper, we are going to assume an even weaker separation condition on the generated self-similar iterated function systems. 
\begin{definition}\label{pm57}
  We say that the self-similar IFS $\mathcal{S}=\{S_k\}_{k=1}^m$ has \textbf{no exact overlapping} if for all $n\geq 1$ and all $\mathbf{i},\mathbf{j}\in[m]^n$ we have 
  \begin{equation}
    S_{\mathbf{i}}\equiv S_{\mathbf{j}} \implies \mathbf{i}=\mathbf{j}.
  \end{equation}
\end{definition}
Examples of IFSs for which the exponential separation condition fails but there are no exact overlappings were given by Baker \cite{baker2021iterated} and Bárány, K{\"a}enm{\"a}ki \cite{barany2021super}.

The main goal of this paper is to prove results on the dependency of the natural pressure on the parameters of a CPLIFS. Our main result is the following. 
\begin{theorem}\label{pm78}
  Let $\mathcal{F}$ be a CPLIFS with generated self-similar IFS $\mathcal{S}$. Suppose that $\mathcal{S}$ has no exact overlapping. Then for every $\varepsilon>0$ there exists a $\delta>0$ such that for all CPLIFS $\widehat{\mathcal{F}}$ which is $\delta$-close to $\mathcal{F}$
  \[
    |s_{\mathcal{F}}-s_{\widehat{\mathcal{F}}}|<\varepsilon,
  \]
  where $s_{\mathcal{F}}$ and $s_{\widehat{\mathcal{F}}}$ are the natural dimensions of $\mathcal{F}$ and $\widehat{\mathcal{F}}$ respectively.
\end{theorem}

As an application of Theorem \ref{pm78}, in section \ref{md75} we prove that under mild conditions the Lebesgue measure of the attractor of a CPLIFS is typically positive if the natural dimension is strictly bigger than $1$.

\begin{theorem}\label{pm50}
  Let $\mathcal{F}$ be a CPLIFS with attractor $\Lambda$ and natural dimension $s$. If the functions of $\mathcal{F}$ only have positive slopes, then for Lebesgue-almost every translation parameters
  \[
    s>1 \implies \mathcal{L}(\Lambda)>0,
  \]
  where $\mathcal{L}$ denotes the appropriate dimensional Lebesgue measure.
\end{theorem}
We state this result rigorously in Theorem \ref{pm77}.

\section{Preliminaries}\label{md22}

\subsection{Parameters of a CPLIFS}\label{pm75}

We fix a number $m\geq 2$, and use it as the number of functions in a CPLIFS throughout the paper. 
Let $\mathcal{F}=\left\{f_k\right\}_{k=1}^{m }$ be a CPLIFS. The functions $f_k:\mathbb{R}\to \mathbb{R}$ are always defined on the whole real line for every $k\in[m]$.
We write $l(k)$ for the number of breaking points  of $f_k$
for $k\in[m]$,
and we say that the \textbf{type} of the CPLIFS is the vector
  \begin{equation}\label{cv44}
  \pmb{\ell }=(l(1), \dots ,l(m)).  
\end{equation}
For example the type of the CPLIFS on Figure \ref{cv47} is
$\pmb{\ell }=(1,2)$. If $\mathcal{F}$ is a CPLIFS of type $\pmb{\ell }$,  then we write $\mathcal{F}\in\mathrm{CPLIFS}_{\pmb{\ell }}$. 

The breaking points of $f_k$ are denoted by $b_{k,1} < \cdots < b_{k,l(k)}$.
Let $L:=\sum_{k=1}^{m}l(k)$ be the total number of breaking points of the functions of $\mathcal{F}$ with multiplicity if some of the breaking points of two different elements of $\mathcal{F} $ coincide.
We arrange all the breaking points in an $L$ dimensional vector $\mathfrak{b}\in \mathbb{R}^L$ as
\begin{equation}\label{cv58}
  \mathfrak{b}
  = (b_{1,1}, \dots ,b_{1,l(1)}, b_{2,1}, \dots ,b_{2,l(2)},
  \dots, b_{m,1}, \dots ,b_{m,l(m)}).
\end{equation}
The set of breaking points vectors $\mathfrak{b}$ for a type $\pmb{\ell }$ CPLIFS is
\begin{equation}\label{cv59}
  \mathfrak{B}^{\pmb{\ell }}:=
  \left\{\mathfrak{b}=(b_{1,1},\dots,b_{m,l(m)})\in\mathbb{R}^{L}
  \bigg\vert\: \forall k\in[m], \forall i\in[l(k)-1]: b_{k,i}<b_{k,i+1} \right\}.
\end{equation}

The $l(k)$ breaking points of the piecewise linear continuous function $f_k$ determines the $l(k)+1$  \textbf{intervals of linearity} $J_{k,i}^{\mathfrak{b}}$, among which the first and the last are actually half lines:
\begin{equation}\label{cv57}
 J_{k,i}:=  J_{k,i}^{\mathfrak{b}}:=
  \left\{
    \begin{array}{ll}
      (-\infty ,b_{k,1}), & \hbox{if $i=1$;} \\
      (b_{k,i-1},b_{k,i}), & \hbox{if $2\leq i\leq l(k)$;} \\
      (b_{k,l(k)},\infty ), & \hbox{if $i=l(k)+1$.}
    \end{array}
  \right.
\end{equation}
The derivative of $f_k$ exists on $J_{k,i}$ and is equal to the constant
\begin{equation}\label{cs68}
  \rho_{k,i}:= f'_k|_{J_{k,i}}.
\end{equation}

We arrange the contraction ratios $\rho_{k,i}\in (-1,1)\setminus \{ 0\}$ into a vector $\pmb{\rho}$, analogously to the breaking points.
\begin{equation}\label{cv56}
  \pmb{\rho}:=\pmb{\rho}_{\mathcal{F}}:= (\rho_{1,1}, \dots ,\rho_{1,l(1)+1},
  \dots,\rho_{m,1}, \dots ,\rho_{m,l(m)+1})\in 
  \left( (-1,1)\setminus \{ 0\}\right)^{L+m}.
\end{equation}
We call $\pmb{\rho}$ the \textbf{vector of contractions}.
The set of all possible values of $\pmb{\rho}$ for an $\mathcal{F}\in \mathrm{CPLIFS}_{\pmb{\ell}}$
is
\begin{equation}\label{cv50}
  \mathfrak{R}^{\pmb{\ell }}:
  =
  \left\{\pmb{\rho}\in\left((-1,1)\setminus \{ 0\}\right)^{L+m} \bigg\vert\:
  \forall k\in[m],\forall i\in[l(k)]: \rho_{k,i}\ne \rho_{k,i+1}
  \right\}.
\end{equation}

Finally, we write
\begin{equation}\label{cv55}
  \tau_k:=f_k(0), \mbox{ and }
  \pmb{\tau}:=(\tau_1, \dots ,\tau_m)\in\mathbb{R}^{m}.
\end{equation}

\begin{figure}[t]
\centering
\includegraphics[width=8cm]{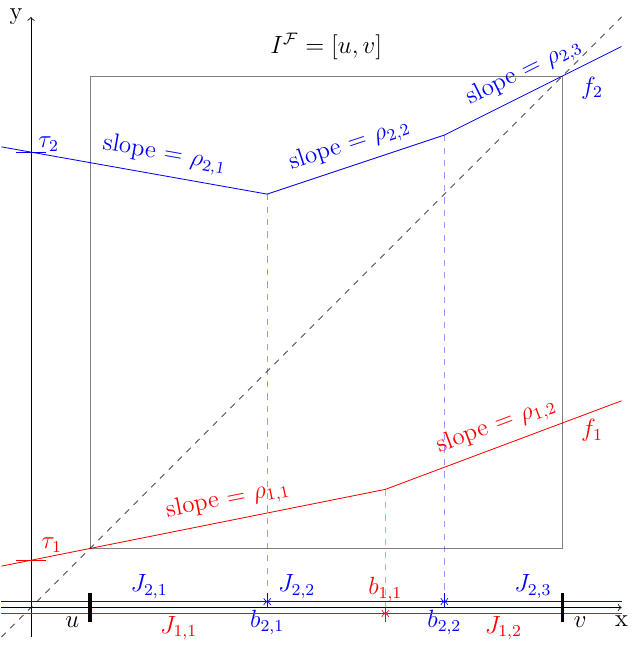}
\caption{A general CPLIFS with the related notations.}\label{cv47}
\end{figure}

So, the parameters that uniquely determine an $\mathcal{F}\in\mathrm{CPLIFS}_{\pmb{\ell}}$ can be organized into a vector
\begin{equation}\label{cv53}
  (\mathfrak{b},\pmb{\tau},\pmb{\rho})\in
  \pmb{\Gamma}^{\pmb{\ell }}:=\mathfrak{B}^{\pmb{\ell}}\times\mathbb{R}^m\times\mathfrak{R}^{\pmb{\ell}}
  \subset \mathbb{R}^L\times\mathbb{R}^m\times\mathbb{R}^{L+m}=
  \mathbb{R}^{2L+2m}.
\end{equation}

For a $(\mathfrak{b},\pmb{\tau},\pmb{\rho})\in \pmb{\Gamma}^{\pmb{\ell}}$ we write $\mathcal{F}^{(\mathfrak{b},\pmb{\tau},\pmb{\rho})}$ for the corresponding CPLIFS, $\Lambda^{(\mathfrak{b},\pmb{\tau},\pmb{\rho})}$ for its attractor, and $s_{(\mathfrak{b},\pmb{\tau},\pmb{\rho})}$ for its natural dimension.

\subsection{Markov diagrams}\label{md21}

Let $\mathcal{F}=\{f_k\}_{k=1}^m$ be a CPLIFS, and let $I$ be the interval defined by \eqref{cr22}. Writing $I_k:=f_k(I)$ and $\mathcal{I}=\cup_{k=1}^m I_k$, we define the \textbf{expanding multi-valued mapping associated to} $\mathcal{F}$ as 
\begin{equation}\label{md59}
    T:\mathcal{I}\mapsto \mathcal{P}(\mathcal{P}(I)),\quad 
    T(y):= \big\{\{x\in I: f_k(x)=y\}\big\}_{k=1}^m .
\end{equation}
That is the image of any Borel subset $A\subset\mathcal{I}$ is 
\[
    T(A)= \big\{\{x\in I: f_k(x)\in A\}\big\}_{k=1}^m.
\]
For $k\in[m], j\in[l(k)+1]$, we define $f_{k,j}:J_{k,j}\to I_k$ as the unique linear function that satisfies $\forall x\in J_{k,j}: f_k(x)=f_{k,j}(x)$.
We call the expansive linear functions 
\begin{align}\label{md58}
    \forall k\in[m], \forall j\in &[l(k)+1]:\quad f_{k,j}^{-1}: f_k(J_{k,j}) \to J_{k,j}, \\
    &\forall x\in J_{k,j}: f_{k,j}^{-1}(f_k(x))=x \nonumber
\end{align}
the \textbf{branches} of the multi-valued mapping $T$. As the notation suggests, these are the local inverses of the elements of $\mathcal{F}$.

\begin{figure}[t]
  \centering
  \includegraphics[width=12cm]{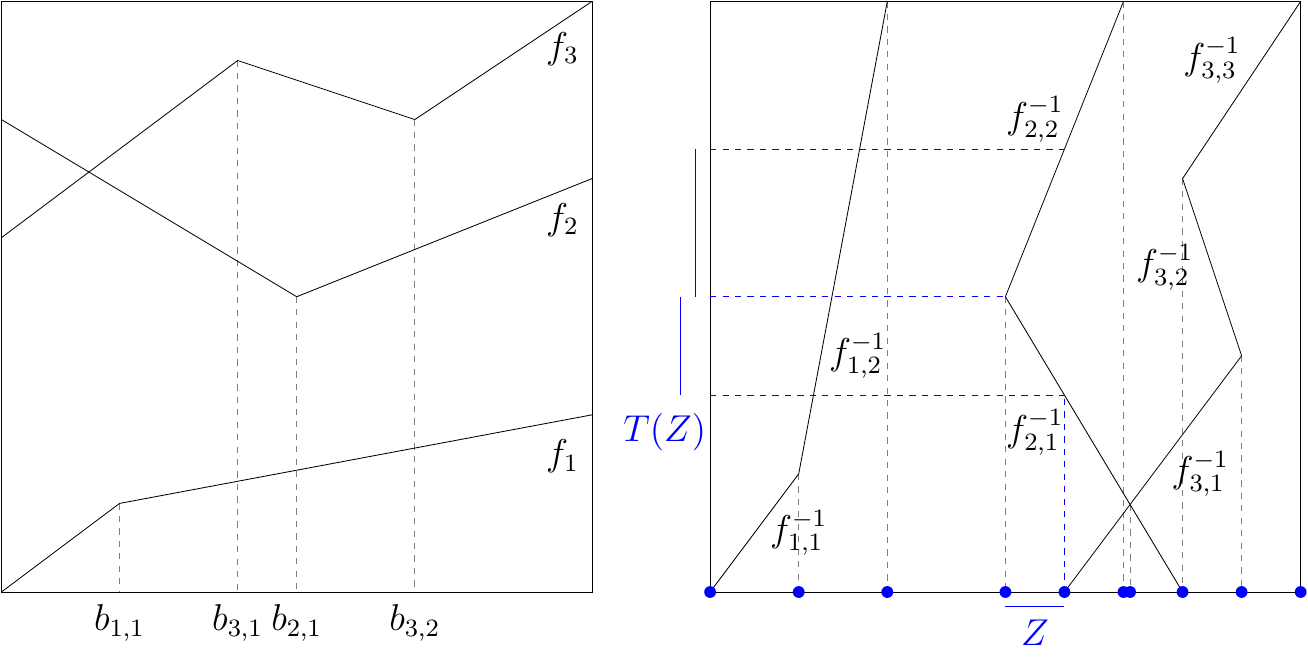}
  \caption{A CPLIFS $\mathcal{F}=\{f_k\}_{k=1}^m$ is on the left with its associated expansive multi-valued mapping $T$ on the right. The critical points are colored with blue.}\label{md25}
\end{figure}

\begin{definition}\label{md57}
    We define the \textbf{set of critical points} as 
      \begin{gather*}
        \mathcal{K}:=\cup_{k=1}^m \{ f_k(0),f_k(1)\} \bigcup 
        \cup_{k=1}^m \cup_{j=1}^{l(k)} f_k(b_{k,j}) \bigcup \\
        \left\{ x\in \mathcal{I} \big\vert \exists k_1,k_2\in [m], 
        \exists j_1\in[l(k_1)], \exists j_2\in[l(k_2)]:  
        f_{k_1, j_1}^{-1}(x)=f_{k_2, j_2}^{-1}(x) \right\}.
      \end{gather*}
\end{definition}

\begin{definition}\label{md56}
    We call the partition of $\mathcal{I}$ into closed intervals defined by the set of critical points $\mathcal{K}$ the \textbf{monotonicity partition} $\mathcal{Z}_0$ of $\mathcal{F}$. We call its elements \textbf{monotonicity intervals}.
\end{definition}

\begin{definition}\label{md55}
    Let $C\subset Z\in \mathcal{Z}_0$ be a closed interval. We say that $D$ is a \textbf{successor} of $C$ and we write $C\to D$ if
    \begin{equation}\label{md54}
      \exists Z_0\in \mathcal{Z}_0, C^{\prime}\in T(C): D=Z_0\cap C^{\prime}.
    \end{equation}
    Further, we write $C\to_{k,j} D$ if 
    \[
      \exists Z_0\in \mathcal{Z}_0: D=Z_0\cap f_{k,j}^{-1}(C).  
    \]
    The set of successors of $C$ is denoted by $w(C):=\{D\vert C\to D\}$.
\end{definition}

Similarly, we define $w(\mathcal{Z}_0)$ as the set of the successors of all elements of $\mathcal{Z}_0$. That is 
\begin{equation}\label{md96}
  w(\mathcal{Z}_0):=\cup_{Z\in \mathcal{Z}_0} w(Z).
\end{equation}

\begin{definition}\label{md95}
  We say that $(\mathcal{D},\to)$ is the \textbf{Markov Diagram of} $\mathcal{F}$ \textbf{ with respect to} $\mathcal{Z}_0$ if $\mathcal{D}$ is the smallest set containing $\mathcal{Z}_0$ such that $\mathcal{D}=w(\mathcal{D})$. For short, we often call it the Markov diagram of $\mathcal{F}$. 
\end{definition}

\begin{remark}\label{md36}
    We can similarly define the Markov diagram of $\mathcal{F}$ with respect to any finite partition $\mathcal{Z}_0^{\prime}$ of $\mathcal{I}$. 
\end{remark}

We often use the notation 

\begin{equation}\label{md60}
  \mathcal{D}_n:=\cup_{i=0}^n w^i(\mathcal{Z}_0) \mbox{, where }
  w^i(\mathcal{Z}_0)=\underbrace{w\circ\cdots\circ w}_{i \mbox{ times}} (\mathcal{Z}_0).
\end{equation}
Obviously, 
\begin{equation}\label{md93}
  \mathcal{D}=\cup_{i\geq 0} w^i(\mathcal{Z}_0).
\end{equation}
If the union in \eqref{md93} is finite, we say that \textbf{the Markov diagram is finite}. We define recursively the $\pmb{n}$\textbf{-th level of the Markov diagram} as
\begin{equation}\label{md53}
   \mathcal{Z}_n:=\omega(\mathcal{Z}_{n-1})\setminus \cup_{i=0}^{n-1} \mathcal{Z}_{i},
\end{equation}
for $n\geq 1$.

One can imagine the Markov diagram as a (potentially infinitely big) directed graph, with vertex set $\mathcal{D}$. Between $C,D\in \mathcal{D}$, we have a directed edge $C\to D$ if and only if $D\in w(C)$. 
A subset $\mathcal{C}\subset\mathcal{D}$ is called \textbf{closed}, if
$C\in\mathcal{C}$, $D\in\mathcal{D}$ and $C\to D$ in
$(\mathcal{D},\to )$ implies $D\in\mathcal{C}$.
We call the Markov diagram \textbf{irreducible} if there exists a directed path between any two intervals $C,D\in \mathcal{D}$. According to the next lemma, we can always assume without loss of generality that the Markov diagram is irreducible. 

\begin{lemma}[{\cite[Lemma~2.6]{prokaj2022fractal}}]\label{md91}
  Let $\mathcal{F}$ be a CPLIFS with attractor $\Lambda$, and let $(\mathcal{D}(\mathcal{Y}_0),\to)$ be its Markov diagram with respect to some finite refinement $\mathcal{Y}_0$ of the monotonicity partition $\mathcal{Z}_0$. For the right choice of $\mathcal{Y}_0$, 
  there exists an irreducible subdiagram $(\mathcal{D}^{'},\to)$ of $(\mathcal{D}(\mathcal{Y}_0),\to)$, such that the elements of $\mathcal{D}^{\prime}$ cover $\Lambda$.
\end{lemma}

\subsection{Connection to the natural pressure}\label{md19}

Similarly to graph-directed iterated function systems, we associate a matrix to Markov diagrams, that will help us determine the Hausdorff dimension of the corresponding CPLIFS (see \cite{falconer1997techniques}).
\begin{definition}\label{md90}
  Let $\mathcal{F}=\{f_k\}_{k=1}^m$ be a CPLIFS, and write $(\mathcal{D},\to )$ for its Markov diagram. We define the matrix $\mathbf{F}(s):=\mathbf{F}_{\mathcal{D}}(s)$ indexed by the elements of $\mathcal{D}$ as
  \begin{equation}\label{md89}
    \left[ \mathbf{F}(s)\right]_{C,D} := \begin{cases}
      \sum_{(k,j): C\to_{(k,j)} D} \vert f_{k,j}^{'}\big\vert^s \mbox{, if $C\to D$} \\
      0 \mbox{, otherwise.}
    \end{cases}  
  \end{equation}
  We call $\mathbf{F}_{\mathcal{D}}(s)$ the \textbf{matrix associated to the Markov diagram} $(\mathcal{D},\to)$.
\end{definition}

We used in the definition, that for a $D\in \mathcal{D}$ with $C\to_{(k,j)} D$ the derivative of $f_{k,j}$ over $D$ is a constant number. That is each element of $\mathbf{F}(s)$ is either zero or a sum of the $s$-th power of some contraction ratios.

This matrix can be defined for any $\mathcal{C}\subset \mathcal{D}$ as well, by choosing the indices from $\mathcal{C}$ only. We write $\mathbf{F}_{\mathcal{C}}(s)$ for such a matrix. It follows that $\mathbf{F}_{\mathcal{C}}(s)$ is always a submatrix of $\mathbf{F}(s)$ for $\mathcal{C}\subset \mathcal{D}$. 

We write $\mathcal{E}_{\mathcal{C}}(n)$ for the set of $n$-length directed paths in the graph $(\mathcal{C},\to)$.
\begin{equation}\label{md87}
  \begin{aligned}
    \mathcal{E}_{\mathcal{C}}(n):= &\{ ((k_1,j_1),\dots ,(k_n,j_n)) \big\vert 
      \exists C_1,\dots ,C_{n+1}\in \mathcal{C}: \\
      &\forall q\in [n] \exists k_q\in[m], j_q\in[l(k)]) : 
      C_{q}\to_{(k_q,j_q)} C_{q+1} \}.
  \end{aligned}
\end{equation}

An $n$-length directed path here means $n$ many consecutive directed edges, and we identify each such path with the labels of the included edges in order.
Each path in $(\mathcal{D},\to)$ of infinite length represents a point in $\Lambda$, and each point is represented by at least one path. 
Similarly, for $\mathcal{C}\subset \mathcal{D}$ the points defined by the natural projection of infinite paths in $(\mathcal{C},\to)$ form an invariant set $\Lambda_{\mathcal{C}}\subset \Lambda$.
We define the natural pressure of these sets as
\begin{equation}\label{md88}
  \Phi_{\mathcal{C}}(s):= \limsup_{n\to \infty}\frac{1}{n}\log
  \sum_{\mathbf{k}} \vert I_{\mathbf{k}}\vert^s,
\end{equation} 
where the sum is taken over all $\mathbf{k}=(k_1,\dots k_n)$ for which $\exists j_1,\dots j_n: ((k_1,j_1),$ $\dots ,(k_n,j_n))\in \mathcal{E}_{\mathcal{C}}(n)$, and
$I$ is the interval defined in \eqref{cr22}.
By the definition of $\mathcal{D}$ it is easy to see that $\Phi_{\mathcal{D}}(s)=\Phi (s)$. 

\begin{remark}\label{md35}
    Let $\mathcal{Y}_0$ be a finite refinement of the monotonicity partition $\mathcal{Z}_0$, and let $(\mathcal{D}^{\prime},\to)$ be the Markov diagram of $\mathcal{F}$ with respect to $\mathcal{Y}_0$. Obviously, 
    \begin{equation}
        \forall s\geq 0: \Phi_{\mathcal{D}}(s)=\Phi_{\mathcal{D}^{\prime}}(s).
    \end{equation}
\end{remark}

We will show, that the unique zero of the function $\Phi_{\mathcal{D}}(s)$ can be approximated by the root of $\Phi_{\mathcal{C}}(s)$ for some $\mathcal{C}\subset \mathcal{D}$. To show this, we need to connect the function $\Phi_{\mathcal{C}}(s)$ to the matrix $\mathbf{F}_{\mathcal{C}}(s)$.

As an operator, $(\mathbf{F}_{\mathcal{D}}(s))^n$ is always bounded in the $l^{\infty}$-norm. Thus we can define 
\[
\varrho (\mathbf{F}_{\mathcal{C}}(s)) := \lim_{n\to \infty} \lVert (\mathbf{F}_{\mathcal{C}}(s))^n\rVert_{\infty} ^{1/n}. 
\] 

\begin{lemma}[{\cite[Lemma~2.9]{prokaj2022fractal}}]\label{md84}
  Let $\mathcal{C}\subset \mathcal{D}$. If $(\mathcal{C},\to)$ is irreducible, then 
  \begin{equation}\label{md86}
    \Phi_{\mathcal{C}}(s)\leq \log \varrho (\mathbf{F}_{\mathcal{C}}(s)).
  \end{equation}
  If $(\mathcal{C},\to)$ is irreducible and finite, then
  \begin{equation}\label{md85}
    \Phi_{\mathcal{C}}(s)= \log \varrho (\mathbf{F}_{\mathcal{C}}(s)).
  \end{equation}
\end{lemma}

Thus if the Markov diagram is finite, the spectral radius of the associated matrix determines the natural pressure of the CPLIFS. We can use this property to calculate the natural dimension of CPLIFSs, that can be approximated by their finite subsystems.

\begin{definition}\label{md82}
  Let $\mathcal{F}$ be a CPLIFS and $\mathcal{Y}$ be a finite refinement of the monotonicity partition $\mathcal{Z}_0$. Let $(\mathcal{D}(\mathcal{Y}),\to)$ be the Markov diagram of $\mathcal{F}$ with respect to $\mathcal{Y}$, and let $\mathbf{F}(\mathcal{Y},s)$ be its associated matrix.

  We say that the CPLIFS $\mathcal{F}$ is \textbf{\nice{}} if there exists a $\mathcal{Y}$ such that for all $s\in(0,s_{\mathcal{F}}]$ the matrix $\mathbf{F}(\mathcal{Y},s)$ has right and left eigenvectors with nonnegative entries for the eigenvalue $\varrho (\mathbf{F}(\mathcal{Y},s))$.

  We call this finite partition $\mathcal{Y}$ a \textbf{limit-irreducible partition of} $\mathcal{F}$ and $(\mathcal{D}(\mathcal{Y}),\to)$ a \textbf{limit-irreducible Markov diagram of} $\mathcal{F}$. 
\end{definition}

\begin{proposition}[{\cite[Proposition~2.11]{prokaj2022fractal}}]\label{md81}
  Let $\mathcal{F}$ be a limit-irreducible CPLIFS, and let $(\mathcal{D},\to)$ be its limit-irreducible Markov diagram.  
  For any $\varepsilon >0$ there exists a $\mathcal{C}\subset \mathcal{D}$ finite subset such that 
  \begin{equation}\label{md80}
    \varrho (\mathbf{F}(s))-\varepsilon \leq \varrho (\mathbf{F}_{\mathcal{C}}(s)) \leq \varrho (\mathbf{F}(s)),
  \end{equation}
  where $\mathbf{F}(s)$ is the matrix associated to $(\mathcal{D},\to)$.
\end{proposition}
  
This proposition and Lemma \ref{md84} together implies the following.
\begin{corollary}\label{pm71}
  Let $\mathcal{F}$ be a limit-irreducible CPLIFS with limit-irreducible Markov diagram $(\mathcal{D},\to)$ and attractor $\Lambda$. Let $\mathbf{F(s)}$ be the matrix associated to $(\mathcal{D},\to)$. Then for all $s\in (0,\dim_{\rm H}\Lambda]$
  \begin{equation}\label{pm99}
    \log\varrho(\mathbf{F}(s)) = \Phi(s).
  \end{equation}
\end{corollary}

With the help of Proposition \ref{md81} and \cite[Corollary~7.2]{prokaj2021piecewise} we also proved the following result on the equality of dimensions of the attractor.
\begin{theorem}[{\cite[Theorem~2.12]{prokaj2022fractal}}]\label{md72}
    Let $\mathcal{F}$ be a \nice{} CPLIFS with attractor $\Lambda$ and 
    \nice{} Markov diagram $(\mathcal{D},\to)$. Assume that the generated self-similar system of $\mathcal{F}$ satisfies the ESC. Then
    \begin{equation}\label{md71}
      \dim_H \Lambda = \min \{1,s_{\mathcal{F}}\},
    \end{equation}
    where $s_{\mathcal{F}}$ denotes the unique zero of the natural pressure function $\Phi(s)$.
\end{theorem}

In \cite[Proposition~3.1]{prokaj2022fractal}, we proved that a CPLIFS $\mathcal{F}$ is limit irreducible if its generated self-similar IFS $\mathcal{S}$ satisfies the ESC. In the proof, we only used the ESC to guarantee that crossing points cannot have periodic orbits starting with a branch that intersects another branch above them. Using the same argument, the next statement follows.

\begin{lemma}\label{pm64}
  Let $\mathcal{F}$ be a CPLIFS with generated self-similar IFS $\mathcal{S}$. If $\mathcal{S}$ has no exact overlappings then $\mathcal{F}$ is limit-irreducible.
\end{lemma}

\begin{proof}
  Let $b\in \mathcal{K}$ be a crossing point, and let $f^{-1}_{k,j_l},f^{-1}_{k,j_r}$ be two branches for which $f^{-1}_{k,j_l}(b)=f^{-1}_{k,j_r}(b)$. Arguing by contradiction, it is enough to show that the existence of a path $b\to_{k,j_l}\to\dots\to b$ in $(\mathcal{G},\to)$ would imply exact overlappings in $\mathcal{S}$.

  Let $n\geq 0$, and assume that for $(k_n,j_n)\dots(k_1,j_1)$ we have 
  \[
    f^{-1}_{(k_n,j_n)\dots(k_1,j_1)(k,j_l)}(b)=b.
  \]
  Writing $\iiv=(k_1,j_1)\dots(k_n,j_n)$, it follows that 
  \begin{equation*}
    S_{k,j_r}\circ S_{\iiv}\circ S_{k,j_l}\circ S_{\iiv} \equiv 
    S_{k,j_l}\circ S_{\iiv}\circ S_{k,j_r}\circ S_{\iiv},
  \end{equation*}
  as these two similarities have the same slope and they take the same value at $b$.
  That is the existence of such a word $(k_n,j_n)\dots(k_1,j_1)$ indeed implies exact overlaps in $\mathcal{S}$.
\end{proof}

\begin{remark}\label{pm63}
  The proof implies that limit-irreducibility is connected to the crossing points of the IFS. If none of the crossing points has a periodic orbit, then the Markov diagram of $\mathcal{F}$ is limit-irreducible. 
\end{remark}

\section{Results on continuity}\label{md24}

Let $\mathcal{F}=\{f_k\}_{k=1}^m$ be a CPLIFS. Recall that the critical points of $\mathcal{F}$ were defined as 
\begin{gather*}
  \mathcal{K}:=\cup_{k=1}^m \{ f_k(0),f_k(1)\} \bigcup 
  \cup_{k=1}^m \cup_{j=1}^{l(k)} f_k(b_{k,j}) \bigcup \\
  \left\{ x\in \mathcal{I} \big\vert \exists k_1,k_2\in [m], 
  \exists j_1\in[l(k_1)], \exists j_2\in[l(k_2)]:  
  f_{k_1, j_1}^{-1}(x)=f_{k_2, j_2}^{-1}(x) \right\}.
\end{gather*}
We call $\mathcal{K}^{\rm int}\subset \mathcal{K}$ the \textbf{set of inner critical points} if it contains all critical points of $\mathcal{F}$ that are interior points of $\mathcal{I}=\cup_{k=1}^m f_k(I)$, where $I$ is the interval defined in \eqref{cr22}.

For technical reasons, we will need to differentiate the endpoints of the monotonicity intervals. As neighbouring intervals share a common endpoint, we introduce a new topological space, where some points of the line are doubled.

Let $\mathcal{Z}_0$ be the monotonicity partition of $\mathcal{F}$, and let $\mathcal{Y}$ be a finite refinement of $\mathcal{Z}_0$. Now write $I=[u,v]$ for the supporting interval of $\mathcal{F}$, and define 
\begin{equation}\label{pm97}
  E:=\{\inf Y, \sup Y: Y\in \mathcal{Y}\}, 
  W:=\left(\bigcup_{j=0}^{\infty} T^{-j}(E\setminus \{u,v\})\right)\setminus\{u,v\},
\end{equation}
where $T^{-1}(A)$ is the preimage of the set $A\subset\mathbb{R}$ by the multi-valued mapping $T$. Observe that $E=\mathcal{K}$ for $\mathcal{Y}=\mathcal{Z}_0$, thus we can say that $E$ takes the role of the set of critical points for finite refinements of $\mathcal{Z}_0$.
We define $\mathbb{R}_{\mathcal{Y}}:=\mathbb{R}\setminus W \bigcup \{x^{-}, x^{+}: x\in W\}$, and set the order $y<x^{-}<x^{+}<z$ if $y<x<z$ holds in $\mathbb{R}$. 

In our new set $\mathbb{R}_{\mathcal{Y}}$, the endpoints of all elements of $\mathcal{Y}$ and their preimages by $T$ are doubled, except $u$ and $v$. We define the projection $\pi:\mathbb{R}_{\mathcal{Y}}\to \mathbb{R}$ as 
\begin{equation}\label{pm96}
  \pi(y)=x \mbox{ if either } y=x\in\mathbb{R} \mbox{ or } y\in\{x^{-}, x^{+}\}.
\end{equation}
This mapping not just connects $\mathbb{R}_{\mathcal{Y}}$ to $\mathbb{R}$, but also preserves the ordering
\[
  y,z\in\mathbb{R}_{\mathcal{Y}}: y<z \implies \pi(y)<\pi(z)
  \mbox{ or } y=x^{-}, z=x^{+} \mbox{ for some } x\in W.
\]
Later it will be useful to jump from one doubled copy to another, so we define the function $\zeta:\pi^{-1}(W)\to \pi^{-1}(W)$
\[
  x\in W: \zeta(x^{+})=x^{-}, \zeta(x^{-})=x^{+}.
\]
For convenience, we extend this map to $\mathbb{R}_{\mathcal{Y}}$ by setting $\zeta(x)=x$ for $x\in\mathbb{R}\setminus \pi^{-1}(W)$.

From now on, we will work on the topological space $\mathbb{R}_{\mathcal{Y}}$ endowed with the order topology. So far we defined every set on $\mathbb{R}$, hence we need to define their counterparts on $\mathbb{R}_{\mathcal{Y}}$. Let $\mathcal{I}_{\mathcal{Y}}$ be the closure of $\mathcal{I}\setminus W$ in $\mathbb{R}_{\mathcal{Y}}$. Observe that $\mathcal{I}_{\mathcal{Y}}$ is compact. We define 
\begin{equation}\label{pm95}
  E_{\mathcal{Y}}:= \{x\in\mathbb{R}_{\mathcal{Y}}: \pi(x)\in E\}\bigcap \mathcal{I}_{\mathcal{Y}}.
\end{equation}
We emphasize that taking the intersection with $\mathcal{I}_{\mathcal{Y}}$ is important, as we are only interested in the endpoints of the monotonicity intervals and not the endpoints of the possible gaps between them.

We write $N:=|\mathcal{Y}|$ for the cardinality of $\mathcal{Y}$ and $E_{\mathcal{Y}}=\{a_1,\dots, a_{2N}\}$ with $a_1<\dots<a_{2N}$. The multi-valued mapping $T|_{\mathcal{I}\setminus W\cup E}$ uniquely extends on $\mathbb{R}_{\mathcal{Y}}$ to a mapping $T_{\mathcal{Y}}$. Similarly, let $\mathcal{K}_{\mathcal{Y}}$ and $\mathcal{K}^{\rm int}_{\mathcal{Y}}$ denote the set of critical points and inner critical points of $T_{\mathcal{Y}}$ respectively. We will supress naming the partition in the lower indices when $\mathcal{Y}=\mathcal{Z}_0$.

Now we define a directed graph $(\mathcal{G},\to)$, different from the Markov diagram, that describes the orbit of the elements of $E_{\mathcal{Y}}$. For $i\in[2N]$ let 
\begin{equation}\label{pm94}
  \omega(i):=\{\iiv=((k_1,j_1),\dots ,(k_n,j_n)) : f^{-1}_{\iiv}(a_i) \in \mathcal{I}_{\mathcal{Y}} \}.
\end{equation}
Note that $\omega(i)$ contains the empty word $\emptyset$ for all $i\in[2N]$, as $a_i\in \mathcal{I}_{\mathcal{Y}}$ for all $a_i\in E_{\mathcal{Y}}$. 
If $\iiv\in\omega(i)$, then $\iiv|_j\in\omega(i)$ as well for all $j\leq|\iiv|$, where $\iiv|_j$ is the word contisting of the first $j$ characters of $\iiv$. 
We define 
\begin{equation}\label{pm93}
  \forall i\in[2N], \forall \iiv\in\omega(i): 
  a_{i,\iiv}:=f^{-1}_{\iiv}(a_i).
\end{equation}

Set $\mathcal{G}:=\{a_{i,\iiv}: a_i\in \mathcal{K}_{\mathcal{Y}}, \iiv\in\omega(i)\}$. For $a,b\in \mathcal{G}$, there is an edge $a\to_{k,j} b$ in the graph $(\mathcal{G},\to)$ if and only if $b=f^{-1}_{k,j}(a)$ or $\zeta(b)=f^{-1}_{k,j}(a)$. We write $a\to b$ if there exists an edge $a\to_{k,j} b$ for some $k\in[m]$ and $j\in[l(k)]$. Observe that this graph does not depend on the partition $\mathcal{Y}$. We call $(\mathcal{G},\to)$ the \textbf{orbit graph of critical points} of $\mathcal{F}$.

Following Definition \ref{md90}, we associate the matrix $\mathbf{G}(s)$, indexed by the elements of $\mathcal{G}$, to the graph $(\mathcal{G},\to)$
\begin{equation}\label{pm92}
  \left[ \mathbf{G}(s)\right]_{a,b} := \begin{cases}
    \sum_{(k,j): a\to_{(k,j)} b} \vert f_{k,j}^{'}\big\vert^s \mbox{, if $a\to b$} \\
    0 \mbox{, otherwise.}
  \end{cases} 
\end{equation}
As $(\mathbf{G}(s))^n$ is always bounded in the $l^{\infty}$-norm, we can define 
\[
\varrho (\mathbf{G}(s)) := \lim_{n\to \infty} \lVert (\mathbf{G}(s))^n\rVert_{\infty} ^{1/n}. 
\]

The following technical lemma will be useful later.
\begin{lemma}\label{pm91}
  Let $\mathcal{F}$ be a CPLIFS with generated self-similar IFS $\mathcal{S}$ and orbit graph $(\mathcal{G},\to)$. Let $\mathbf{F}(s)$ and $\mathbf{G}(s)$ be the matrices associated to the Markov diagram and the orbit graph of critical points respectively. If $\mathcal{S}$ doesn't have exact overlappings, then 
  \begin{equation}\label{pm90}
    \varrho(\mathbf{G}(s)) \leq \varrho(\mathbf{F}(s)). 
  \end{equation} 
\end{lemma}

\begin{proof}
  Fix a path of infinite length $d_1\to d_2\to\dots$ in $(\mathcal{G},\to)$. Let $J\subset\mathbb{N}$ be the set of indices such that 
  \[
    \forall i\in J:\quad d_i\in\{ \zeta(d): d\in T_{\mathcal{Y}}(d_{i-1})\}.
  \]
  Note that $i\in J$ implies $d_i\in \mathcal{K}^{\rm int}$.
  If $J=\emptyset$, then there exists a corresponding path in $(\mathcal{D},\to)$. Namely, the path $D_1\to D_2\to\dots$, where for all $i\geq 1, d_i$ is an endpoint of $D_i$.
  
  Clearly, \eqref{pm90} can fail only if there is a path in $(\mathcal{G},\to)$ that has bigger weight than any constant multiple of the weight of any path in $(\mathcal{D},\to)$. It can only happen if $J$ is an infinite set. As $|\mathcal{K}^{\rm int}|<\infty$, it is equivalent to a $d\in \mathcal{K}^{\rm int}$ having a periodic orbit.
  However, a periodic orbit in general does not imply $\varrho(\mathbf{G}(s))>\varrho(\mathbf{F}(s))$.

  The elements of $\mathcal{K}^{\rm int}$ fall into three categories: images of breaking points, crossing points, and endpoints of overlapping first cylinders.
  It is easy to see that $\varrho(\mathbf{G}(s))>\varrho(\mathbf{F}(s))$ can only happen if the following is satisfied
  \begin{align}\label{pm66}
    \exists d\in \mathcal{K}^{\rm int},\exists k\in[m],\exists j_l,j_r\in &[l(k)+1]: \\ f^{-1}_{k,j_l}(d)=f^{-1}_{k,j_r}(d), &\mbox{ and } \exists\: d\to_{k,j_l}\dots\to d \mbox{ path in } (\mathcal{G},\to). \nonumber
  \end{align}
  That is we must have a $d\in\mathcal{K}^{\rm int}$ which is an image of a breaking point and has a periodic orbit starting with the right branch. 

  Assume now that there exists a $d\in \mathcal{K}^{\rm int}$ and branches $(k,j_l), (k,j_r)$ satisfying \eqref{pm66}. It follows that for some $n\geq 0$
  \begin{equation}\label{pm65}
    \exists \;\iiv=((k_1,j_1),\dots,(k_n,j_n)): 
    S_{\iiv (k,j_l)}(d) = d. 
  \end{equation}
  Since $\mathcal{S}=\{S_{k,j}\}_{k\in[m],j\in[l(k)]}$ is a self-similar IFS, \eqref{pm65} implies exact overlappings in the system. Namely, the functions $S_{\iiv (k,j_l) \iiv (k,j_r)}$ and $S_{\iiv (k,j_r) \iiv (k,j_l)}$ are identical.

\end{proof}

\subsection{Continuity of the Markov diagram}

We want to show that the natural dimension of a given CPLIFS does not change too much if we perturb its parameters. To do this, we need to define what we mean by two CPLIFSs being close to each other.
\begin{definition}\label{pm82}
  Let $[c,d], [\widehat{c},\widehat{d}]\subset \mathbb{R}$ be two closed intervals. We say that they are $\pmb{\varepsilon}$\textbf{-close} if 
  \[
    |c-\widehat{c}|<\varepsilon \mbox{ and } |d-\widehat{d}|<\varepsilon.
  \] 
\end{definition}

\begin{definition}\label{pm89}
  Let $\mathcal{Y}=\{Y_1,\dots,Y_N\}$ and $\widehat{\mathcal{Y}}=\{\widehat{Y}_1,\dots,\widehat{Y}_N\}$ be finite partitions.
  We say that they are $\pmb{\varepsilon}$\textbf{-close} if for all $j\in[N], Y_j$ and $\widehat{Y}_j$ are $\varepsilon$-close. 
\end{definition}

\begin{definition}\label{pm83}
  We say that the CPLIFSs $\mathcal{F}=\{f_k\}_{k=1}^m$ and $\widehat{\mathcal{F}}=\{\widehat{f}_k\}_{k=1}^m$ are $\pmb{\varepsilon}$\textbf{-close} if 
  \begin{enumerate}[{\bf (a)}]
    \item their monotonicity partitions are $\varepsilon$-close,
    \item $\forall k\in[m]: f_k$ \mbox{ and } $\widehat{f}_k$ has the same number of breaking points,
    \item $\forall k\in[m], \forall j\in[l(k)]: \left|\log |f^{\prime}_{k,j}|-\log |\widehat{f}^{\prime}_{k,j}|\right|<\varepsilon.$ 
    \item $\forall k\in[m]: \|\widehat{f}_k-f_k\|_{\infty}<\varepsilon$
  \end{enumerate}
\end{definition}

Let $\mathcal{F}$ be a CPLIFS with associated multi-valued mapping $T$ and monotonicity partition $\mathcal{Z}_0$. Let $\mathcal{Y}$ be a finite refinement of $\mathcal{Z}_0$ and $(\mathcal{D},\to)$ be the Markov diagram of $\mathcal{F}$ with respect to $\mathcal{Y}$. We define 
\begin{equation}\label{pm88}
  \mathcal{M}:=\{(i,\iiv): i\in[2N], \iiv\in\omega(i)\},
\end{equation}
and $\mathcal{M}_r:=\{(i,\iiv)\in \mathcal{M}: |\iiv|\leq r\}$ for $r\in \mathbb{N}_0$. P. Raith \cite[p~108]{raith1992continuity} proved that we can think of the Markov diagram as the image of $\mathcal{M}$ under a special mapping.
While \cite{raith1992continuity} only investigates the case when $T$ is an expansive mapping of the line and not multi-valued, the mapping $A$ constructed in the proof works in our case as well, since we differentiate between the images generated by the different branches.  

\begin{lemma}\label{pm87}
  There exists a mapping $A:\mathcal{M}\to \mathcal{D}$ with the following properties.
  \begin{enumerate}[{\bf (a)}]
    \item $A(\mathcal{M})=\mathcal{D}$ and $A(\mathcal{M}_r)=\mathcal{D}_r$ for all $r\in\mathbb{N}_0$.
    \item $a_i$ is an endpoint of $A(i,\emptyset)$ for all $i\in[2N]$.
    \item $a_{i,\iiv}$ is an endpoint of $A(i,\iiv)$ for all $(i,\iiv)\in\mathcal{M}$.
    \item Let $k\in[m]$ and $j\in[l(k)]$. We introduce a graph structure on $\mathcal{M}$ such that $c\to_{(k,j)} d$ in $\mathcal{M}$ implies $A(c)\to_{(k,j)} A(d)$ in $\mathcal{D}$. For every $c\in\mathcal{M}$ the map $A$ is bijective from $\{d\in\mathcal{M}: c\to_{(k,j)} d\}$ to $\{D\in\mathcal{D}: A(c)\to_{(k,j)} D\}$.
    \item $c\in \mathcal{M}_r$ implies the existence of a $d\in \mathcal{M}_r$ with $A(c)\subset A(d)$ and either $A(c)=[a_d,a_c]$ or $A(c)=[a_c,a_d]$. 
  \end{enumerate}
\end{lemma}
This map $A$ will be surjective, but need not be injective. A $C\in\mathcal{D}$ can be represented by multiple elements of $\mathcal{M}$. It implies that there might be several subsets in $\mathcal{M}$ whose image under $A$ is $\mathcal{D}$. Each of them can take the role of the Markov diagram if they satisfy the following definition.

\begin{definition}\label{pm86}
  $(\mathcal{A},\to)$ is called a \textbf{variant of the Markov diagram} of $\mathcal{F}$ with respect to $\mathcal{Y}$ if $\mathcal{A}\subset \mathcal{M}$ satisfies the following properties for all $k\in[m]$ and $j\in[l(k)]$.
  \begin{enumerate}[{\bf (a)}]
    \item For $i\in[2N]$ and $\iiv\in\omega(i)$, $(i,\iiv)\in \mathcal{A}$ implies $(i,\iiv|_j)\in \mathcal{A}$ for all $j\leq |\iiv|$.
    \item $c,d\in \mathcal{A}$ and $c\to_{(k,j)} d$ in $\mathcal{M}$ imply $c\to_{(k,j)} d$ in $\mathcal{A}$. 
    \item $c,d\in \mathcal{A}$ and $c\to_{(k,j)} d$ in $\mathcal{A}$ imply either $c\to_{(k,j)} d$ in $\mathcal{M}$ or there exists a $d_0\in \mathcal{M}\setminus\mathcal{A}$ with $c\to_{(k,j)} d_0$ and $A(d)=A(d_0)$.
    \item For $c\in\mathcal{A}$ the map $A:\{d\in\mathcal{A}: c\to_{(k,j)} d\} \to \{D\in\mathcal{D}: A(c)\to_{(k,j)} D\}$ is bijective.
    \item $A(\mathcal{A}\cap\mathcal{M}_r)=\mathcal{D}_r$ for $r\in\mathbb{N}_0$.
  \end{enumerate}
\end{definition}

Observe that $(\mathcal{M},\to)$ and $(\mathcal{D},\to)$ are also variants of the Markov diagram of $\mathcal{F}$ with respect to $\mathcal{Y}$. For $r\in\mathbb{N}_0$ set $\mathcal{A}_r:=\mathcal{A}\cap\mathcal{M}_r$.

Let $(\mathcal{A},\to)$ be a variant of the Markov diagram. We write $\mathbf{F}^{\mathcal{A}}(s)$ for the matrix associated to $(\mathcal{A},\to)$, and define it analogously to \ref{md89}. The following lemma is a straightforward consequence of Definition \ref{pm86}.

\begin{lemma}\label{pm84}
  Let $\mathcal{F}$ be a CLPIFS, and let $\mathcal{Y}$ be a finite refinement of its monotonicity partition $\mathcal{Z}_0$. Let $(\mathcal{A},\to), (\mathcal{A}^{\prime},\to)$ be two variants of the Markov diagram of $\mathcal{F}$ with respect to $\mathcal{Y}$. Then 
  \[
    \varrho(\mathbf{F}^{\mathcal{A}}(s)) = \varrho(\mathbf{F}^{\mathcal{A}^{\prime}}(s)).
  \]
\end{lemma}

\begin{remark}\label{pm73}
  Let $(\mathcal{A},\to)$ be a variant of the Markov diagram, and fix $\varepsilon>0$. Assume that there exists a finite irreducible $\mathcal{C}\subset \mathcal{D}$ such that $\varrho (\mathbf{F}_{\mathcal{C}}(s)) > \varrho (\mathbf{F}(s))-\varepsilon$. Then obviously $\exists r: \mathcal{C}\subset \mathcal{D}_r$, and by part $v)$ of Definition \ref{pm86}, 
  \[
    \exists \widetilde{\mathcal{C}}\subset \mathcal{A}_r \mbox{ such that }
    \varrho (\mathbf{F}^{\mathcal{A}}_{\widetilde{\mathcal{C}}}(s)) > \varrho (\mathbf{F}(s))-\varepsilon.
  \]
\end{remark}

Raith proved \cite[Lemma~6]{raith1992continuity} that if two systems are close to each other, then the initial parts of their Markov diagrams coincide. Our systems has a more complicated structure due to the possible overlappings, but the construction given in his proof works here also.

Let $\mathcal{Y}=\{Y_1,\dots,Y_N\}$ be a finite partition. For $1<i<N$, we say that $C\subset\mathbb{R}$ is $\mathcal{Y}$-close to $Y_i$ if $C\subset Y_{i-1}\cup Y_i\cup Y_{i+1}$.

\begin{lemma}[{\cite[Lemma~6]{raith1992continuity}}]\label{pm85}
  Let $\mathcal{F}$ be a CLPIFS, and let $\mathcal{Y}$ be a finite refinement of its monotonicity partition. Then for every $r\in\mathbb{N}$ there exists a $\delta>0$ such that for every CPLIFS $\widehat{\mathcal{F}}$ which is $\delta$-close to $\mathcal{F}$ and for every finite refinement $\widehat{\mathcal{Y}}$ of $\widehat{\mathcal{Z}_0}$ which is $\delta$-close to $\mathcal{Y}$, there exists a variant $(\mathcal{A},\to)$ of the Markov diagram of $\mathcal{F}$ with respect to $\mathcal{Y}$ and a variant $(\widehat{\mathcal{A}},\to)$ of the Markov diagram of $\widehat{\mathcal{F}}$ with respect to $\widehat{\mathcal{Y}}$ with the following properties.
  \begin{enumerate}[{\bf (i)}]
    \item $\widehat{\mathcal{A}}_r$ can be written as a disjoint union $\mathcal{B}_0\cup \mathcal{B}_1\cup \mathcal{B}_2$, such that $\mathcal{B}_1\cup \mathcal{B}_2$ and $\mathcal{B}_2$ are closed in $\widehat{\mathcal{A}}_r$ and $\widehat{\mathcal{A}}_0\subset \mathcal{B}_0$ ($\mathcal{B}_1$ and $\mathcal{B}_2$ might be empty).
    \item Every $c\in\widehat{\mathcal{A}}_r$ has at most two successors in $\mathcal{B}_1\cup \mathcal{B}_2$ by a given branch.
    \item There exists a bijective function $\phi:\mathcal{A}_r\to \mathcal{B}_0$, and there exists a function $\psi:\mathcal{B}_2\to \mathcal{G}$.
    \item For $c,d\in \mathcal{A}_r$ and $k\in[m],j\in[l(k)]$ the property $c\to_{(k,j)}d$ in $\mathcal{A}$ is equivalent to $\phi(c)\to_{(k,j)}\phi(d)$ in $\widehat{\mathcal{A}}$. For $c,d\in \mathcal{B}_2$ the property $c\to_{(k,j)}d$ in $\widehat{\mathcal{A}}$ implies $\psi(c)\to_{(k,j)}\psi(d)$ in $\mathcal{G}$.
    \item $A(c)=Y_j$ for $c\in \mathcal{A}_0$ implies $\phi(c)\in \widehat{\mathcal{A}}_0$ and $\widehat{A}(\phi(c))=\widehat{Y}_j$.
    \item $c\in \mathcal{A}_r$ and $A(c)\subset A(d)$ for a $d\in \mathcal{A}_0$ imply $\widehat{A}(\phi(c))\subset \widehat{A}(\phi(d))$. $c\in \mathcal{B}_2$ and $\psi(c)\in A(d)$ for a $d\in \mathcal{A}_0$ imply $\widehat{A}(c)$ is $\widehat{Y}$-close to $\widehat{A}(\phi(d))$.
    \item Let $\mathcal{P}$ be the set of all paths $c_0\to c_1\to\dots \to c_r$ of length $r$ in $\widehat{A}_r$ with $c_0\in \widehat{\mathcal{A}}_0$, and set $\mathcal{N}:=\{(d_0,\dots ,d_r): d_j\in \mathcal{A}_r\cup \mathcal{G} \mbox{ for }  j\in\{0,\dots ,r\}\}$. Then there exists a function $\chi: \mathcal{P}\to \mathcal{N}$.
    \item  Let $c_0\to c_1\to\dots \to c_r\in \mathcal{P}, \chi(c_0\to c_1\to\dots \to c_r)=(d_0,\dots ,d_r)$ and $j\in\{0,1\dots ,r\}$. $c_j\in \mathcal{B}_2$ is equivalent to $d_j\in \mathcal{G}$, and we have then $\psi(c_j)=d_j$. $c_j\in \mathcal{B}_0$ implies $\phi(d_j)=c_j$. $c_j\in \mathcal{B}_0\cup \mathcal{B}_1$ implies $\widehat{A}(c_j)$ is $\widehat{\mathcal{Y}}$-close to $\widehat{A}(\phi(d_j))$. Moreover, $c_j\in \mathcal{B}_0\cup \mathcal{B}_1$ implies $d_{j-1}\to d_j$ in $\mathcal{A}$ for $j\geq 1$.
    \item For a fixed $c_0\in\widehat{\mathcal{A}}_0$, a fixed $(d_0,\dots ,d_r)\in \mathcal{N}$ and for a fixed sequence of branches $(k_1,j_1),\dots,(k_r,j_r)$ there are at most $2r+1$ different paths $\mathbf{c}:=c_0\to_{(k_0,j_0)} \cdots\to_{(k_r,j_r)} c_r\in \mathcal{P}$ such that $\chi(\mathbf{c})=(d_0,\dots ,d_r)$. Further, for $q\in\{0,1,\dots ,r-1\}$ and fixed $d_0,d_1,\dots ,d_q\in \mathcal{A}_r\cap \mathcal{G}$ there are at most $4$ different $a\in \mathcal{G}$ such that there exists $d_{q+2},d_{q+3},\dots ,d_r\in \mathcal{A}\cup \mathcal{G}$ with $(d_0,d_1,\dots ,d_q,a,d_{q+2},d_{q+3},\dots ,d_r)\in\chi(\mathcal{P})$.
  \end{enumerate}
\end{lemma}

In \cite{raith1992continuity} a constructive proof is given. The same construction also works in our case. Note that in contrast with \cite{raith1992continuity}, our $T$ is multi-valued, this is why we have to specify which branches we use in parts $(ii)$ and $(ix)$. As this construction is very long, technical and presents nothing new we omit to repeat it here.

Instead, we give a general explanation on the statements to help building intuition on the structure of the Markocv diagrams.
Using the lemma's notations, it states that for any big integer $r\in\mathbb{N}$ we can find a small number $\delta>0$ such that the $r$-th level of the variants $\mathcal{A}_r$ and $\widehat{\mathcal{A}}_r$ have similar structures, given that $\mathcal{F}$ and $\widehat{\mathcal{F}}$ are $\delta$-close to each other.

Namely, $\widehat{A}_r$ contains a bijective copy of $\mathcal{A}_r$ which we denote by $\mathcal{B}_0$. According to (iv), there is a one-to-one correspondence between the edges of $(\mathcal{A}_r,\to)$ and $(\mathcal{B}_0,\to)$. As (v) and (vi) implies, the same relation is present between the intervals corresponding to the elements of $\mathcal{A}_r$ and $\mathcal{B}_0$. That is, a big initial part of the diagram of $\mathcal{F}$ is contained in the diagram of $\widehat{\mathcal{F}}$.

Assume that $\mathcal{F}$ has an inner critical point with a periodic orbit. 
Since perturbing the parameters of the original IFS may result in destroying this periodic orbit, we cannot expect $\widehat{\mathcal{F}}$ to share the same Markov diagram with $\mathcal{F}$, which is a representation of the orbit structure of these dynamical systems. In this case $\widehat{\mathcal{A}}_r\setminus \mathcal{B}_0$ is non-empty, this is the set we call $\mathcal{B}_1\cup \mathcal{B}_2$. 

Let $d\in\widehat{\mathcal{A}}_r\setminus \mathcal{B}_0$ be a code which appears in the variant $\widehat{\mathcal{A}}_r$ but isn't included in $\mathcal{A}_r$. It must code the successor of a point that is mapped onto an $a_l\in E_{\mathcal{Y}}$ by the expansive mapping $T$, for some $l\in[2N]$. Critical points have assigned codes, $a_l$ is already coded by $(l,\emptyset)$, that is why $d$ is not in $\mathcal{A}_r$.

If $a\in \mathcal{K}_{\mathcal{Y}}$, then $d\in \mathcal{B}_2$. By the definition of the orbit graph of critical points $(\mathcal{G},\to)$, we can relate the elements of $\mathcal{B}_2$ to the nodes of $\mathcal{G}$. It also suggests that $(\mathcal{B}_2,\to)$ is a closed subgraph of $(\widehat{\mathcal{A}}_r,\to)$. 

If $a\in E_{\mathcal{Y}}\setminus\mathcal{K}_{\mathcal{Y}}$, then we say $d\in \mathcal{B}_1$. These points move along with the endpoint of an element of the partition $\mathcal{Y}$. Since this endpoint is not a critical point, we can always find an element of $\mathcal{B}_0$ which is mapped onto the same monotonicity interval by $\widehat{A}$, according to the last two lines of (viii).

\subsection{Continuity of the natural pressure}

We prove here that the natural dimension of a limit-irreducible CPLIFS is always  lower semi-continuous, and we give a condition for its upper semi-continuity as well. The proofs follow the lines of the proof of \cite[Theorem~1]{raith1992continuity} and \cite[Theorem~2]{raith1992continuity} respectively.

\begin{theorem}\label{pm80}
  Let $\mathcal{F}$ be a limit-irreducible CPLIFS with attractor $\Lambda$, and let $s\in(0,\dim_{\rm H}\Lambda)$. Then for every $\varepsilon>0$ there exists a $\delta>0$ such that for all CPLIFS $\widehat{\mathcal{F}}$ which is $\delta$-close to $\mathcal{F}$
  \[
    \Phi(s)-\varepsilon < \widehat{\Phi}(s),
  \]
  where $\Phi(s)$ and $\widehat{\Phi}(s)$ stand for the natural pressure function of $\mathcal{F}$ and $\widehat{\mathcal{F}}$ respectively.
\end{theorem}

\begin{proof}
  Let $\mathcal{Y}$ be the limit-irreducible partition of $\mathcal{F}$, and let $(\mathcal{D},\to)$ be the Markov diagram of $\mathcal{F}$ with respect to $\mathcal{Y}$. 

  Fix $\varepsilon>0$ and write $K$ for the maximum number of overlapping branches. By Proposition \ref{md81}, there exists a $\mathcal{C}\subset \mathcal{D}$ finite subset such that
  \begin{equation}\label{pm72}
    \varrho(\mathbf{F}(s))-\frac{\varepsilon}{2} \leq 
    \varrho(\mathbf{F}_{\mathcal{C}}(s)),
  \end{equation}
  where $\mathbf{F}(s)$ is the matrix associated to the Markov diagram. As $\mathcal{C}$ is finite it must be contained in $\mathcal{D}_r$ for some $r\geq 1$. By Lemma \ref{pm85}, there exists a $\delta\in(0,\frac{\varepsilon}{2sK})$ such that the conclusions of Lemma \ref{pm85} are true with respect to $r$ if $\widehat{\mathcal{F}}$ is $\delta$-close to $\mathcal{F}$. 

  Let $\widehat{\mathcal{Z}}_0$ be the monotonicity partition of $\widehat{\mathcal{F}}$ and $\widehat{\mathcal{Y}}$ be a finite refinement of $\widehat{\mathcal{Z}}_0$ which is $\delta$-close to $\mathcal{Y}$. Let $(\mathcal{A},\to), (\widehat{\mathcal{A}},\to)$ be the variants of the Markov diagram of $\mathcal{F}, \widehat{\mathcal{F}}$ with repsect to $\mathcal{Y}$ and $\widehat{\mathcal{Y}}$ respectively, obtained by applying Lemma \ref{pm85}.  

  By \eqref{pm72} and Remark \ref{pm73}, there exists a $\mathcal{C}\subset \mathcal{A}_r$ such that 
  \begin{equation}\label{pm70}
    \Phi(s)-\varepsilon = \log\varrho (\mathbf{F}(s))-\varepsilon
    \leq \log\varrho (\mathbf{F}^{\mathcal{A}}_{\mathcal{C}}(s))-\frac{\varepsilon}{2},
  \end{equation}
  where the first equality follows from Corollary \ref{pm71}. 
  
  We write $\rho$ and $\widehat{\rho}$ for the slopes of the functions of $\mathcal{F}$ and $\widehat{\mathcal{F}}$ respectively. 
  Consider the matrix $\mathbf{F}^{\widehat{\mathcal{A}}}_{\phi(\mathcal{C})}(s)$, where $\phi: \mathcal{A}_r\to \widehat{\mathcal{A}}_r$ is the mapping described in Lemma \ref{pm85}. As $\mathcal{F}$ and $\widehat{\mathcal{F}}$ are $\delta$-close, for any $k\in[m]$ and $j\in[l(k)]$ we have 
  \[
    \log \rho_{k,j}-\delta\leq \log\widehat{\rho}_{k,j}.
  \] 
  Using that every entry in $\mathbf{F}^{\mathcal{A}}_{\mathcal{C}}(s)$ is a sum of at most $K$ many elements of the form $\rho^{s}_{k,j}$, by parts $(iv), (v)$ and $(vi)$ of Lemma \ref{pm85} we have the following relation between the elements of $\widehat{\mathbf{F}}^{\widehat{\mathcal{A}}}_{\phi(\mathcal{C})}(s)$ and $\mathbf{F}^{\mathcal{A}}_{\mathcal{C}}(s)$.
  \begin{equation}\label{pm69}
    \forall c,d\in \mathcal{C}:\quad  
    \left[\widehat{\mathbf{F}}^{\widehat{\mathcal{A}}}_{\phi(\mathcal{C})}(s)
    \right]_{\phi(c),\phi(d)} \geq 
    \exp^{-\frac{\varepsilon}{2}}\left[\mathbf{F}^{\mathcal{A}}_{\mathcal{C}}(s) 
    \right]_{c,d},
  \end{equation} 
  since $\delta$ is smaller than $\frac{\varepsilon}{2sK}$.
  By \eqref{pm70}, \eqref{pm69} and Lemma \ref{md84}, we conclude the proof 
  \[
    \Phi(s)-\varepsilon\leq \log\varrho (\mathbf{F}^{\mathcal{A}}_{\mathcal{C}}(s))-\frac{\varepsilon}{2}\leq \log\varrho (\widehat{\mathbf{F}}^{\widehat{\mathcal{A}}}_{\phi(\mathcal{C})}(s)) = \widehat{\Phi}_{\phi(\mathcal{C})}(s)\leq \widehat{\Phi}(s).
  \]

\end{proof}

\begin{theorem}\label{pm79}
  Let $\mathcal{F}$ be a limit-irreducible CPLIFS with attractor $\Lambda$, and let $\mathbf{G}(s)$ be the matrix defined in \eqref{pm92}. Fix an arbitrary $s\in(0,\dim_{\rm H}\Lambda)$. Then for every $\varepsilon>0$ there exists a $\delta>0$ such that for all CPLIFS $\widehat{\mathcal{F}}$ which is $\delta$-close to $\mathcal{F}$
  \[
    \widehat{\Phi}(s) < \max\{\Phi(s), 
    \log \varrho(\mathbf{G}(s))\} +\varepsilon,
  \]
  where $\Phi(s)$ and $\widehat{\Phi}(s)$ stand for the natural pressure function of $\mathcal{F}$ and $\widehat{\mathcal{F}}$ respectively.
\end{theorem}

The proof is analogous to the proof of \cite[Theorem~2]{raith1992continuity}.

\begin{proof}
    Let $\mathcal{Y}$ be the limit-irreducible partition of $\mathcal{F}$, and let $(\mathcal{D},\to)$ be the Markov diagram of $\mathcal{F}$ with respect to $\mathcal{Y}$. 

    Fix $\varepsilon>0$ and write $K$ for the maximum number of overlapping branches. We define the value 
    \[
      R_0:= \exp\left( \max\left\{
        \Phi(s), \log\varrho(\mathbf{G}(s))
      \right\} +\varepsilon\right).  
    \]
    As $R_0>e^{\frac{\varepsilon}{2}}\max\left\{e^{\Phi(s)},\varrho(\mathbf{G}(s))\right\}$, we can choose an
    \[
      R\in \left(e^{\frac{\varepsilon}{2}}\max\left\{e^{\Phi(s)},\varrho(\mathbf{G}(s))\right\}, R_0\right).
    \]
    It follows that $e^{-\frac{\varepsilon}{2}}R>\varrho(\mathbf{G}(s))$ and $e^{-\frac{\varepsilon}{2}}R>e^{\Phi(s)}$. By Lemma \ref{pm84} and Corollary \ref{pm71}, it also follows that for any $(\mathcal{A},\to)$ variant of the Markov diagram of $\mathcal{F}$ with respect to $\mathcal{Y}$, we have $e^{-\frac{\varepsilon}{2}}R>\varrho (\mathbf{F}^{\mathcal{A}}(s))$.
    
    By Gelfand's formula, for any bounded linear operator $\mathbf{M}$ the spectral radius satisfies $\varrho(\mathbf{M})=\inf_{n\geq 1} \Vert \mathbf{M}^n\Vert^{\frac{1}{n}}$. Thus, there exists a $C\in \mathbb{R}$ such that 
    \begin{equation}\label{pm61}
        \sup_{n\in \mathbb{N}} e^{\frac{\varepsilon}{2}n}R^{-n}
        \Vert \mathbf{G}(s)^n\Vert \leq C, 
    \end{equation}
    and for every $(\mathcal{A},\to)$ variant of the Markov diagram of $\mathcal{F}$ with respect to $\mathcal{Y}$
    \begin{equation}\label{pm60}
        \sup_{n\in \mathbb{N}} e^{\frac{\varepsilon}{2}n}R^{-n}
        \Vert \mathbf{F}^{\mathcal{A}}(s)^n\Vert \leq C.
    \end{equation}
    We may assume that 
    \begin{equation}
        C\geq \max\left\{ 2, 8K^{r-1}e^{\frac{\varepsilon}{2}}R^{-1}
        \rho_{\max}^s\right\},
    \end{equation}
    where $\rho_{\max}=\max_{k\in[m], j\in[l(k)]} \vert\rho_{k,j}\vert$ is the biggest slope of the functions in $\mathcal{F}$ in absolute value. Fix an $r\in \mathbb{N}$ with 
    \begin{equation}
        \sqrt[r]{(r+1)^2C^3}R<R_0.
    \end{equation}

    By Lemma \ref{pm85}, there exists a $\delta\in(0,\frac{\varepsilon}{2})$ such that the conclusions of Lemma \ref{pm85} are true with respect to $r$ if $\widehat{\mathcal{F}}$ is $\delta$-close to $\mathcal{F}$.

    Let $\widehat{\mathcal{Z}}_0$ be the monotonicity partition of $\widehat{\mathcal{F}}$ and $\widehat{\mathcal{Y}}$ be a finite refinement of $\widehat{\mathcal{Z}}_0$ which is $\delta$-close to $\mathcal{Y}$. Let $(\mathcal{A},\to), (\widehat{\mathcal{A}},\to)$ be the variants of the Markov diagram of $\mathcal{F}, \widehat{\mathcal{F}}$ with repsect to $\mathcal{Y}$ and $\widehat{\mathcal{Y}}$ respectively, obtained by applying Lemma \ref{pm85}.

    Like in the proof of Theorem \ref{pm80}, we write $\rho$ and $\widehat{\rho}$ for the slope of the functions of $\mathcal{F}$ and $\widehat{\mathcal{F}}$ respectively. Let $\phi:\mathcal{A}_r\to \widehat{\mathcal{A}}_r$ be the bijective mapping described in Lemma \ref{pm85}. As $\mathcal{F}$ and $\widehat{\mathcal{F}}$ are $\delta$-close, for any $k\in[m]$ and $j\in[l(k)]$ we have 
    \begin{equation}\label{pm59}
      \log\widehat{\rho}_{k,j}\leq\log\rho_{k,j}+\delta\leq\log\rho_{k,j}+\frac{\varepsilon}{2}.
    \end{equation}

    Let $\mathcal{P}_c$ be the set of all paths $c_0\to_{(k_0,j_0)} c_1\to_{(k_1,j_1)} \cdots\to_{(k_{r-1},j_{r-1})} c_r$ of length $r$ in $\widehat{\mathcal{A}}_r$ with $c_0=c$.
    To prove the statement of the theorem, it is enough to give an upper bound on the sum of the weights of paths of length $r$ in $\widehat{\mathcal{A}}_r$ starting at the same node, since
    \begin{gather}\label{pm58}
      \left(\varrho(\widehat{\mathbf{F}}(s))\right)^r \leq 
      \Vert \widehat{\mathbf{F}}(s)^r\Vert_{\infty} =
      \Vert \widehat{\mathbf{F}}^{\widehat{\mathcal{A}}_r}(s)^r\Vert_{\infty} \\
      =\sup_{c\in \widehat{\mathcal{A}}_0} \sum_{\mathcal{P}_c} \prod_{n=0}^{r-1} \vert\widehat{\rho}_{k_n,j_n}\vert^{s}, \nonumber
    \end{gather}
    where the sum is taken over all paths $c_0\to_{(k_0,j_0)} c_1\to_{(k_1,j_1)} \cdots\to_{(k_{r-1},j_{r-1})} c_r$ in $\mathcal{P}_c$. 
    
    Fix $c\in\widehat{\mathcal{A}}_0$. Then by (i) and (iii) of Lemma \ref{pm85}, there exists a unique $d\in \mathcal{A}_0$ with $\phi(d)=c$. 
    Using the notation of Lemma \ref{pm85}, we write $\widehat{\mathcal{A}}_r=\mathcal{B}_0\cup\mathcal{B}_1\cup\mathcal{B}_2$ such that $\mathcal{B}_1\cup\mathcal{B}_2$ and $\mathcal{B}_2$ are closed in $\widehat{\mathcal{A}}_r$. Then, we can partition the paths in $\mathcal{P}_c$ based on the index of the last node in $\mathcal{B}_0\cup\mathcal{B}_1$. 
    For $q\in\{0,1\dots,r\}$, let $\mathcal{P}_c(q)$ be the set of all paths $c_0\to c_1\to\cdots\to c_r\in\mathcal{P}_c$ with $q=\max\{j\in\{0,1\dots,r\}: c_j\in\mathcal{B}_0\cup\mathcal{B}_1\}$. 
    Thus $\mathcal{P}_c=\cup_{s=0}^r\mathcal{P}_c(q)$. Define for $q\in\{0,1\dots,r\}$
    \begin{equation}\label{pm56}
      H_c(q):=\sum_{\mathcal{P}_c(q)}\prod_{n=0}^{r-1} \vert\widehat{\rho}_{k_n,j_n}\vert^{s}.
    \end{equation}
    Writing $H_c:=\sum_{q=0}^{r}H_c(q)$, we can reformulate \eqref{pm58} as
    \begin{equation}\label{pm52}
      \left(\varrho(\widehat{\mathbf{F}}(s))\right)^r \leq\sup_{c\in \widehat{\mathcal{A}}_0} H_c.
    \end{equation}

    Let $q\in\{0,1,\dots,r\}$. If $\mathbf{c}:=c_0\to_{(k_0,j_0)} c_1\to_{(k_1,j_1)} \cdots\to_{(k_{r-1},j_{r-1})} c_r\in\mathcal{P}_(c)$ and $(d_0,d_1,\dots,d_r)=\chi(\mathbf{c})$, then (i) and (viii) of Lemma \ref{pm85} give $d_j\in\mathcal{A}_r$ for all $j\leq q$, $d_0=d$ and $d_j\in\mathcal{G}$ for all $j>q$. 
    Hence if $q\geq 1$, then $d_0\to_{(k_0,j_0)} d_1\to_{(k_1,j_1)} \cdots\to_{(k_{q-1},j_{q-1})} d_q$ is a path of length $q$ in $\mathcal{A}_r$ with $d_0=d$. 
    Moreover, (viii) of Lemma \ref{pm85} also guarantees that $\widehat{A}(c_j)$ is $\widehat{\mathcal{Y}}$-close to $\widehat{A}(\phi(d_j))$ for all $j\in\{0,1,\dots,q\}$. 
    Thus by \eqref{pm59}
    \begin{equation}\label{pm55}
      \prod_{n=0}^{q-1}\vert\widehat{\rho}_{k_n,j_n}\vert^s \leq
      e^{\frac{\varepsilon}{2}sq} \prod_{n=0}^{q-1}\vert\rho_{k_n,j_n}\vert^s,
      \quad \mbox{if } q\geq 1.
    \end{equation}

    If $q\leq r-2$, then by (iv), (vi) and (viii) of Lemma \ref{pm85}, the path $d_{q+1}\to_{(k_{q+1},j_{q+1})}$ $\cdots\to_{(k_{r-1},j_{r-1})}d_r$ is of length $r-q-1$ in $\mathcal{G}$, and $\widehat{A}(c_j)$ is $\widehat{\mathcal{Y}}$-close to $\widehat{A}(\phi(\tilde{d}_j))$ for all $j\in\{q+1,q+2,\dots,r\}$, where $\tilde{d}_j\in\mathcal{A}_0$ satisfies $d_j\in A(\tilde{d}_j)$. Hence \eqref{pm59} gives 
    \begin{equation}\label{pm54}
      \prod_{n=q+1}^{r-1}\vert\widehat{\rho}_{k_n,j_n}\vert^s \leq
      e^{\frac{\varepsilon}{2}s(r-q-1)} \prod_{n=q+1}^{r-1}\vert\rho_{k_n,j_n}\vert^s,
      \quad \mbox{if } q\leq r-2.
    \end{equation}

    We apply (ix) of Lemma \ref{pm85} to obtain for $1\leq q\leq r-2$
    \begin{align*}
      H_c(q) &\leq K^{r-1}(2r+1)e^{\frac{\varepsilon}{2}sr}
      \sum_{(d_0,d_1,\dots,d_r)\in\chi(\mathcal{P}_c(q))} \vert\rho_{k_q,j_q}\vert^s
      \prod_{n=0}^{q-1}\vert\rho_{k_n,j_n}\vert^s
      \prod_{n=q+1}^{r-1}\vert\rho_{k_n,j_n}\vert^s \\
      &\leq K^{r-1}8e^{\frac{\varepsilon}{2}sr}\rho_{\max}^s(r+1)
      \left(\sum_{d=d_0\to_{(k_0,j_0)}\cdots\to_{(k_{q-1}j_{s-1})}d_q} 
      \prod_{n=0}^{q-1}\vert\rho_{k_n,j_n}\vert^s\right) \\
      &\cdot\left(\sup_{a\in\mathcal{G}} \sum_{a_0=a\to_{(\tilde{k}_0,\tilde{j}_0)}\cdots\to_{(\tilde{k}_{r-q-2},\tilde{j}_{r-q-2})}a_{r-q-1}} 
      \prod_{n=0}^{r-q-2}\vert\rho_{\tilde{k}_n,\tilde{j}_n}\vert^s\right).
    \end{align*}
    Keep in mind that according to ix) of Lemma \ref{pm85}, at most $2r+1$ paths has the same image by $\chi$, but only if they use the same edges. Due to the possible overlappings in $\mathcal{F}$ and $\widehat{\mathcal{F}}$, there might be paralel edges in the Markov diagrams, so it is possible that we can walk the same path using completely different edges. That is why we need the $K^{r-1}$ multiplier.

    Recall that $C\geq 8K^{r-1}e^{\frac{\varepsilon}{2}sr}R^{-1}\rho_{\max}^s$. 
    By \eqref{pm61} and \eqref{pm60} we obtain 
    \begin{equation}\label{pm53}
      H_c(q)\leq CR(r+1)CR^qCR^{r-q-1}=(r+1)C^3R^r. 
    \end{equation}
    Similarly, the same bound holds for the remaining three values of $q$
    \begin{align*}
      H_c(0) &\leq K^{r-1}8e^{\frac{\varepsilon}{2}sr}\rho_{\max}^s(r+1)
      \left(\sup_{a\in\mathcal{G}} \sum_{a_0=a\to_{(\tilde{k}_0,\tilde{j}_0)}\cdots\to_{(\tilde{k}_{r-2},\tilde{j}_{r-2})}a_{r-1}} 
      \prod_{n=0}^{r-2}\vert\rho_{\tilde{k}_n,\tilde{j}_n}\vert^s\right) \\
      &\leq CR(r+1)CR^{r-1}\leq (r+1)C^3R^r, \\
      H_c(r-1) &\leq K^{r-1}8e^{\frac{\varepsilon}{2}sr}\rho_{\max}^s(r+1)
      \left(\sum_{d=d_0\to_{(k_0,j_0)}\cdots\to_{(k_{r-2}j_{r-2})}d_{r-1}} 
      \prod_{n=0}^{r-2}\vert\rho_{k_n,j_n}\vert^s\right) \\
      &\leq CR(r+1)CR^{r-1}\leq (r+1)C^3R^r, \\
      H_c(r) &\leq K^{r-1}2e^{\frac{\varepsilon}{2}sr}(r+1) 
      \left(\sum_{d=d_0\to_{(k_0,j_0)}\cdots\to_{(k_{r-1}j_{r-1})}d_r} 
      \prod_{n=0}^{r-1}\vert\rho_{k_n,j_n}\vert^s\right) \\
      &\leq C(r+1)CR^r\leq (r+1)C^3R^r.
    \end{align*}
    
    Thus, $H_c=\sum_{q=0}^r H_c(q)\leq (r+1)^2C^3R^r$ for all $c\in\widehat{\mathcal{A}}_0$. Then \eqref{pm52} gives 
    \begin{equation}\label{pm51}
      \varrho(\widehat{\mathbf{F}}(s))\leq \sqrt[r]{(r+1)^2C^3}\cdot R<R_0.
    \end{equation}
    We conclude the proof by combining \eqref{pm51} with Lemma \ref{md84}
    \begin{equation*}
      \widehat{\Phi}(s)\leq \log\varrho(\widehat{\mathbf{F}}(s)) <
      \log R_0=\max\{\Phi(s),\log\varrho(\mathbf{G}(s))\}+\varepsilon.
    \end{equation*}
      
\end{proof}

We have seen in Lemma \ref{pm64} that a CPLIFS with no exact overlapping is limit-irreducible.
Therefore, as a consequence of Theorem \ref{pm80}, Theorem \ref{pm79} and Lemma \ref{pm91}, Theorem \ref{pm78} follows.

\begin{remark}\label{pm62}
  By the proof of Lemma \ref{pm91}, it also follows that the natural dimension of a CPLIFS changes continuously with respect to its parameters if the critical points which are breaking point images does not have periodic orbits.
\end{remark}

That is the natural dimension of a CPLIFS $\mathcal{F}$ changes continuously with respect to the parameters that define $\mathcal{F}$ (breaking points, slopes of functions, translation parameters) if there are no exact overlappings. This result combined with Theorem \ref{md72} yields an analogous result on the Hausdorff dimension of the attractor.

\begin{corollary}
  Let $\mathcal{F}$ be a CPLIFS whose generated self-similar IFS satisfies the ESC.
  Fix an arbitrary $\varepsilon>0$, and let $\mathfrak{P}$ be the property that the CPLIFS $\mathcal{F}^{\prime}$ satisfies  
  \begin{equation}
    |\dim_{\rm H}\Lambda -\dim_{\rm H}\widehat{\Lambda}|<\varepsilon ,
  \end{equation}
  where $\Lambda$ and $\widehat{\Lambda}$ are the attractors of $\mathcal{F}$ and $\widehat{\mathcal{F}}$ respectively.

  Then there exists a $\delta>0$ such that $\mathfrak{P}$ is a $\dim_{\rm P}$-typical property of CPLIFSs that are $\delta$-close to $\mathcal{F}$. 
\end{corollary}

We close this section with an example demonstrating that the natural dimension is not necessarily continuous with respect to the parameters of a CPLIFS.
\begin{example}\label{pm49}
    Pick an arbitrary $\varepsilon>0$, and let us define the following piecewise linear functions.
    \begin{equation*}
        f_1(x)=\begin{cases}
            \frac{2}{5}x,\: x<0 \\
            \frac{1}{5}x,\: x\geq 0
        \end{cases}
        \hspace{-10pt},\quad
        f_2(x) =\frac{1}{3}x,\quad
        \widehat{f}_2(x) =\frac{1}{3}x +\varepsilon
    \end{equation*}
    Consider the CPLIFSs $\mathcal{F}=\{f_1,f_2\}$ and $\widehat{\mathcal{F}}=\{f_1,\widehat{f}_2\}$. Write $\Lambda, \widehat{\Lambda}$ for the attractors and $\Phi(s), \widehat{\Phi}(s)$ for the natural pressure functions of $\mathcal{F}$ and $\widehat{\mathcal{F}}$ respectively. These two iterated function systems are Clearly $\varepsilon$-close to each other.

    It is easy to see that $\widehat{\Lambda}$ is a Cantor set. Since $\mathcal{F}$ satisfies the strong separation property, $\widehat{s}_{\rm nat}=\dim_{\rm H}\widehat{\Lambda}$ is the unique solution of 
    \[
      \left(\frac{1}{3}\right)^s+\left(\frac{1}{5}\right)^s=1. 
    \] 

    As $0$ is the fixed point of both $f_1$ and $f_2$, the supporting interval $I$ of $\mathcal{F}$ is $[-\frac{1}{2},\frac{1}{2}]$. The slope of $f_1$ is strictly bigger than $\frac{1}{5}$ over $[-\frac{1}{2},0]$, so $s_{\rm nat}$ is expected to be strictly bigger than $\widehat{s}_{\rm nat}$. 

    Set $\Sigma:=\{1,2\}^{\mathbb{N}}$, and let $\#_2(\iiv)$ be the number of $2$-s in $\iiv\in\Sigma$. Observe that for any $\iiv\in\Sigma$ the length of $I_{\iiv}$ only depends on $\#_2(\iiv)$ and not the position of the $2$-s.
    A short calculation gives 
    \[
      \left| I_{\iiv}\right| = \left(\frac{1}{3}\right)^{n-k} 
      \frac{2^k+1}{2\cdot 5^k},\quad \mbox{ if } \#_2(\iiv)=k.  
    \]
    The natural pressure function of $\widehat{\mathcal{F}}$ is 
    \begin{align*}
      \widehat{\Phi}(s) &= \limsup_{n\to\infty}\frac{1}{n}\log\sum_{\substack{\iiv\in\Sigma :\\ |\iiv|=n}} \left|I_{\iiv}\right|^s =
      \limsup_{n\to\infty}\frac{1}{n}\log\sum_{k=0}^n\sum_{\substack{\iiv\in\Sigma :\\ |\iiv|=n \:{\rm and}\: \#_2(\iiv)=k}} \left|I_{\iiv}\right|^s =\\
      &= \limsup_{n\to\infty}\frac{1}{n}\log\sum_{k=0}^n {n\choose k}
      \left(\frac{1}{3}\right)^{s(n-k)}\left(\frac{2^k+1}{2\cdot 5^k}\right)^s 
      > \log\left( \left(\frac{1}{3}\right)^s+\left(\frac{1}{5}\right)^s \right).
    \end{align*}
    Thus, $\widehat{s}_{\rm nat} < s_{\rm nat}$ independently of $\varepsilon$.
\end{example}    

\section{Lebesgue measure of the attractor}\label{md75}

\begin{theorem}\label{pm77}
  Fix a type $\pmb{\ell}$ and a vector of slopes $\pmb{\rho}\in\mathfrak{R}^{\pmb{\ell}}$. If all elements of $\pmb{\rho}$ are positive, then for $\mathcal{L}^{L+m}$-almost every $(\mathfrak{b},\pmb{\tau})\in\mathfrak{B}^{\pmb{\ell}}\times\mathbb{R}^m$
  \begin{equation}\label{pm76}
    s_{(\mathfrak{b},\pmb{\tau},\pmb{\rho})}>1 \implies 
    \mathcal{L}\left( \Lambda^{(\mathfrak{b},\pmb{\tau},\pmb{\rho})}\right) >0,
  \end{equation}
  where $s_{(\mathfrak{b},\pmb{\tau},\pmb{\rho})}$ and $\Lambda^{(\mathfrak{b},\pmb{\tau},\pmb{\rho})}$ are the natural dimension and attractor of the CPLIFS defined by the parameters $(\mathfrak{b},\pmb{\tau},\pmb{\rho})$.
\end{theorem}

\begin{proof}
  Let $E_{\rm ESC}\subset\mathfrak{B}^{\pmb{\ell}}\times\mathbb{R}^m$ be the set of parameters for which the generated self-similar IFS $\mathcal{S}^{(\mathfrak{b},\pmb{\tau},\pmb{\rho})}$ satisfies the ESC. Hochman proved \cite[Theorem~1.10]{hochman2014self} that 
  \[
    \mathcal{L}^{L+m}(\mathfrak{B}^{\pmb{\ell}}\times\mathbb{R}^m\setminus E_{\rm ESC}) =0.
  \]
  Thus it is enough to focus on the elements of $E_{\rm ESC}$. Fix an arbitrary $(\mathfrak{b},\pmb{\tau})\in E_{\rm ESC}$ and assume that $s_{(\mathfrak{b},\pmb{\tau},\pmb{\rho})}>1$. Set $\varepsilon:=\Phi_{(\mathfrak{b},\pmb{\tau},\pmb{\rho})}(1)$. The natural pressure function is strictly decreasing, and $s_{(\mathfrak{b},\pmb{\tau},\pmb{\rho})}>1$, hence $\varepsilon>0$. 

  Since our parameter space $\mathfrak{B}^{\pmb{\ell}}\times\mathbb{R}^m \subset \mathbb{R}^{m+L}$ is $\sigma$-compact, it suffices to show that \eqref{pm76} holds for $\mathcal{L}^{L+m}$-almost every element of a well defined open neighbourhood around any point $(\mathfrak{b},\pmb{\tau})\in \mathfrak{B}^{\pmb{\ell}}\times\mathbb{R}^m$. 
  As a CPLIFS whose generated self-similar IFS satisfies the ESC is always limit-irreducible, according to Theorem \ref{pm80}, there exists a $B((\mathfrak{b},\pmb{\tau}), \delta)\subset\mathfrak{B}^{\pmb{\ell}}\times\mathbb{R}^m$ open ball around $(\mathfrak{b},\pmb{\tau})$ of radius $\delta$ such that for all $(\widehat{\mathfrak{b}},\widehat{\pmb{\tau}})\in B((\mathfrak{b},\pmb{\tau}), \delta)$
  \[
    \forall s\in[0,s_{(\mathfrak{b},\pmb{\tau},\pmb{\rho})}) :\quad 
    \Phi_{(\mathfrak{b},\pmb{\tau},\pmb{\rho})}(s)-\varepsilon < \Phi_{(\widehat{\mathfrak{b}},\widehat{\pmb{\tau}},\pmb{\rho})}(s).
  \]
  It follows that 
  \begin{equation}\label{pm68}
    \forall (\widehat{\mathfrak{b}},\widehat{\pmb{\tau}})\in B((\mathfrak{b},\pmb{\tau}), \delta): \quad
    s_{(\widehat{\mathfrak{b}},\widehat{\pmb{\tau}},\pmb{\rho})}>1.
  \end{equation}

  Write $\mathbf{t}_{(\widehat{\mathfrak{b}},\widehat{\pmb{\tau}})}$ for the vector of translations of $\mathcal{S}^{(\widehat{\mathfrak{b}},\widehat{\pmb{\tau}},\pmb{\rho})}$. By \cite[Theorem~1]{keane2003dimension}, for $\mathcal{L}^{L+m}$-almost every $\mathbf{t}_{(\widehat{\mathfrak{b}},\widehat{\pmb{\tau}})}\in\mathbb{R}^{L+m}$ the assertion of the theorem holds for $\mathcal{S}^{(\widehat{\mathfrak{b}},\widehat{\pmb{\tau}},\pmb{\rho})}$. By \cite[Fact~4.1]{prokaj2021piecewise}, \eqref{pm76} holds for  $\mathcal{L}^{L+m}$-almost every $(\widehat{\mathfrak{b}},\widehat{\pmb{\tau}})\in B((\mathfrak{b},\pmb{\tau}), \delta)$. As $(\mathfrak{b},\pmb{\tau})$ was arbitrary, the assertion of the theorem follows.

\end{proof}

\bibliographystyle{abbrv}
\bibliography{CPLIFS_bib}

\end{document}